\documentclass[11pt]{amsart}
\usepackage{amsmath}
\usepackage{amssymb}
\usepackage{graphicx}
\newtheorem{theorem}{Theorem}[section]
\newtheorem{corollary}[theorem]{Corollary}
\newtheorem{lemma}[theorem]{Lemma}
\newtheorem{proposition}[theorem]{Proposition}
\theoremstyle{definition}
\newtheorem{definition}[theorem]{Definition}
\newtheorem{example}[theorem]{Example}

\numberwithin{equation}{section}

\newcommand{\RR}{{\mathbb R}}
\newcommand{\NN}{{\mathbb N}}

\newcommand{\WEff}{{\mbox{\rm WEff\,}}}
\newcommand{\Eff}{{\mbox{\rm Eff\,}}}

\newcommand{\rint}{{\mbox{\rm int\,}}}

\newcommand{\diam}{{\rm diam\,}}

\providecommand{\norm}[1]{\left\lVert#1 \right\rVert}

\title[Genericity of well-posed vector optimization problems]{Genericity of well-posed vector optimization problems }

\author[Matteo Rocca]{Matteo Rocca}

\address[M. Rocca]{Department of Economics, Universit\'a degli Studi dell'Insubria, Via Monte Generoso 71, 21100 Varese, Italy}
\email{{\tt matteo.rocca@uninsubria.it}}

\keywords{vector optimization, well-posedness, scalarization}

\subjclass[2010]{90C25, 90C26, 90C29}

\begin{document}

\begin{abstract}
In this paper we consider well-posedness properties of vector optimization problems with objective function $f: X \to Y$ where  $X$ and $Y$ are Banach spaces and  $Y$ is partially ordered by a closed convex pointed cone with nonempty interior.  The vector well-posedness notion considered in this paper is the one due to Dentcheva and Helbig \cite{dh}, which is a natural extension of Tykhonov well-posedness for scalar optimization problems. When a scalar optimization problem is considered it is possible to prove (see e.g. \cite{lucchetti2}) that under some assumptions the set of functions for which the related optimzation problem is well-posed is dense or even more in "big" i.e. contains a dense $G_{\delta}$ set (these results are called genericity results). The aim of this paper is to extend these genericity results to vector optimzaition problems.  
\end{abstract}

\maketitle




\section{Introduction}
Well-posedness properties are important qualitative characterizations for scalar and vector optimization problems.  In particular, the notion of well-posedness plays a central role in stability theory
for scalar optimization (see e.g. \cite{dz}). 
The well-posedness notion for scalar functions dates back to  Hadamard  \cite{hadamard} and to  Tykhonov in \cite{tykhonov}. Extensions to vector and set-valued cases are presented in several papers and are still a topic of research (see e.g. \cite{b1}, \cite{b2},\cite{loridan}, \cite{dh}, \cite{dz}, \cite{lucchetti2}, \cite{mmr}, \cite{huang1}, \cite{huang2}, \cite{gmmn}, \cite{ckr}). In \cite{mmr}  a classification of vector well-posedness notions in two classes is given: pointwise and global notions. 
The definitions of the first  group consider a fixed  efficient  point (or the image of an efficient point) and deal with well-posedness of the vector optimization problem at this point. This approach imposes that the minimizing sequences related to the considered point are well-behaved. Since in the vector case the solution set is typically not a singleton, there is also a class of  definitions, the so-called global notions, that involve the efficient frontier as a whole. \ \\
In the scalar optimization,  a  crucial point is the identification of classes of  objective functions for which the related optimization problem enjoys well-posedness properties. It is known  (see e.g. \cite{dz})   that scalar optimization problems with convex objective function  $f:X \to \RR$, $X$ finite-dimensional,  enjoy well-posedness properties. Similarly, it is known that vector optimzation problems with  cone-convex objective function functions $f: X \to Y$ with $X$  and $Y$ finite-dimensional, enjoy well posedness properties (see. eg. \cite{mmr}).\ \\
 When functions $f: X \to \RR$ with $X$ infinite-dimensional are considered, it is known that convexity does not guarantee well-posedness (see e.g. \cite{dz}).  In this case it is interesting to find sets   of functions for which the subset of  well-posed functions  is dense (when a suitable topology on  the considered  set of functions is introduced). A stronger version of these results leads to find sets of functions for which the subset of well-posed funcions is "big" in the sense of Baire category, i.e. contains a dense $G_{\delta}$ set (see e.g. \cite{il},  \cite{lucchetti2}  and references therein). \ \\
The aim of this paper is to  extend this kind of results, called genericity results,  to vector optimzation problems with objective function  $f: X \to Y$ where $X$ and $Y$ are Banach spaces. In our investigation we will focus on the  pointwise well-posedness notion for vector functions due to Dentcheva and Helbig \cite{dh}. \ \\
The outline of the paper is the following. In  Section \ref{preliminaries} we introduce the notation and we recall some preliminary notions. In Section \ref{wp}  we recall some scalar and vector well-posedness notions. In Section \ref{density} we give results concerning density of well-posed vector  optimization problems, without convexity assumptions.  Section \ref{Genericity} is devoted to genericity results under cone-convexity assumptions.

\section{Preliminaries}\label{preliminaries}

Let $X$ and $Y$ be Banach spaces, $f:X \to Y$  and $C \subseteq Y$ a closed, convex, pointed cone with nonempty interior, endowing $Y$ with a partial order in the following way

\begin{equation}
\begin{split}
y_1\le_C y_2 & \iff y_2-y_1\in C\\
y_1<_C y_2 & \iff y_2-y_1\in {\rm int}\,  C
\end{split}
\end{equation}

In the following for a set $A \subseteq X$ we denote by ${\rm diam}\, A$ the diameter of $A$, i.e. 
$$
{\rm diam}\, A= {\rm sup}\{\|x-y\|: \ x, y \in A\}
$$
We denote respectively by $B$  the closed unit ball   both in $X$ and $Y$  (from the context it will be clear to which space we refer). \ \\ We denote by $Y^*$ the topological dual space of $Y$ and by $C^*$ the positive polar cone of $C$, i.e.
\begin{equation*}
C^*=\{\xi\in Y^*: \langle \xi , c\rangle \ge 0, \ \forall c \in C\}
\end{equation*}
Consider  the vector optimization problem
\begin{equation}\tag{$X, f$}\label{vp}
\min f(x), \ \ x \in X.
\end{equation}
A point $\bar{x} \in X$ is called an   efficient solution for problem (\ref{vp}) when
\begin{equation*}
(f(X) - f(\bar{x})) \cap (-C) = \{ 0 \}
\end{equation*}
We denote by $\Eff(X,f)$ the set of all efficient solutions for problem  (\ref{vp}). 
A point $\bar{x} \in X$ is called a weakly efficient solution for problem  (\ref{vp}) when
\begin{equation*}
(f(X) - f(\bar{x})) \cap (-\rint C) = \emptyset.
\end{equation*}
We denote by $\WEff(X,f)$ the set of weakly efficient solutions for problem  (\ref{vp}). 
We recall also (see e.g.\cite{b3}) that a point $\bar x \in X$ is said to be a strictly efficient solution for problem (\ref{vp}) when, for every $\varepsilon>0$, there exists $\delta>0$ such that
\begin{equation}
(f(X)-f(\bar x))\cap (\delta B -C)\subseteq \varepsilon B
\end{equation}
We denote by ${\rm StEff}(X, f)$ the set of strictly efficient solutions for problem (\ref{vp}). Clearly  ${\rm StEff}(X, f)\subseteq \Eff(X,f) \subseteq {\rm WEff}(X,f)$.

\begin{definition} \cite{luc}
A function $f: X \to Y$,  is said to be $C-$convex if $\forall x,z \in X$ and $t \in [0,1]$ it holds 
\begin{equation*}
f(tx+(1-t)z) \in tf(x) + (1-t)f(z)-C
\end{equation*}
\end{definition}

\begin{proposition}(see e.g. \cite{luc})
$f:X \to Y$ is $C$-convex  if and only if functions $g_{\xi}(x)= \langle \xi , f(x)\rangle$ are convex for every $\xi \in C^*$. 
\end{proposition}

We recall also that a function $f: X \to Y$ is said to be $*$-quasiconvex  when functions $g_{\xi}(x)= \langle \xi , f(x)\rangle$ are  quasiconvex for every $\xi \in C^*$ (see e.g. \cite{jeya}). \ \\
 For $y \in Y$,  $L_f^C(y) := \{x \in X : f(x) \in y -C\}$  is the level set of $f$. 
We say that $f:X \to Y$ is $C$-lower semicontinuous ($C$-lsc for short) when $L_f^C(y)$ is closed for every $y \in Y$ \cite{luc}. \ \\
Now, we recall, the notion of oriented distance between a point $y \in Y$  and a set $A\subseteq Y$, denoted by $D_A(y)$. 

\begin{definition}
For a set $A\subseteq  Y$ the oriented distance is the function $D_A : Y \to \RR\cup \{\pm \infty\}$
defined as
\begin{equation}
D_A(y) = d_A(y) - d_{Y \backslash A}(y)
\end{equation}
with $d_{\emptyset}(y) = +\infty$.
\end{definition}

Function $D_A$ was introduced in \cite{hiriart1979, hiriart1979a} to analyze the geometry of nonsmooth optimization problems and obtain necessary optimality conditions. The next result summarizes some basic properties of function $D_A$. 

\begin{proposition} \cite{zaffaroni, ggr}\label{properties} If the set $A$ is nonempty and $A \not = Y$ , then
\begin{itemize}
\item [1.] $D_A$ is real valued;
\item [2.] $D_A$ is $1$-Lipschitzian;
\item [3.] $D_A(y) < 0$ for every $y \in  {\rm int}\, A$, $D_A(y) = 0$ for every $y \in  \partial A$ and $D_A(y) > 0$
for every $y \in  {\rm int}\, (Y\backslash A)$;
\item [4.] if $A$ is closed, then it holds $A = \{y : D_A(y)  \le 0\}$;
\item [5.] if $A$ is convex, then $D_A$ is convex;
\item [6.] if $A$ is a cone, then $D_A$ is positively homogeneouos;
\item [7.] if $A$ is a closed convex cone, then $D_A$ is nonincreasing with respect to the
ordering relation induced on $Y$, i.e. the following is true: if $y_1,  y_2 \in  Y$ then
$y_1 - y_2 \in  A \Rightarrow D_A(y_1) \le  D_A(y_2)$; if $A$ has nonempty interior, then
$y_1 - y_2 \in  {\rm int}\, A \Rightarrow D_A(y_1) < D_A(y_2)$:
\item [8.] It holds
\begin{equation}
D_A(y)=\max_{\xi\in C^*\cap \partial B}\langle \xi , y\rangle 
\end{equation}
where $\partial A$ denotes the boundary of the set $A$. 
\end{itemize}
\end{proposition}

\begin{theorem}\cite{mmr}
If  $f:X \to Y$ is $C$-convex, then for every $y \in Y$, function $D_{-C}(f(x)-y)$ is convex. 
\end{theorem}
We associate to problem (\ref{vp}) the scalar problem 

\begin{equation*}\label{sp}\tag{$X, D_{-C}$}
\min D_{-C}(f(x)-f(\bar x))\ , \ x \in X
\end{equation*}
with $\bar x \in X$. 
The relations of the solutions of this problem with those of problem (\ref{vp} ) are investigated in \cite{zaffaroni, mmr, ggr}. For the convenience of the reader, we quote the characterization of efficient  points and weakly efficient  points.
\begin{theorem}\cite{zaffaroni, mmr, ggr}\label{min-car}
Let $f: X \to Y$. 
\begin{itemize}
\item [1.] $\bar x \in \WEff(X, f)$ if and only if $\bar x$ is a solution of problem (\ref{sp}). 
\item [2.] If  $\bar x$ is the unique solution of problem (\ref{sp}), then $ \bar x \in \Eff(X, f)$.  
\end{itemize}
\end{theorem}

\section{Well-posedness for scalar and vector optimization problems}\label{wp}
\subsection{Well-posedness for scalar optimization problems}

In this section we recall the notion of  well-posednsess for  functions $f:X\to \RR$ introduced by Tykhonov \cite{tykhonov}. For  a complete treatment of this notion and of its generalizations one can refer to \cite{dz, lucchetti2}.  
Clearly in this case problem (\ref{vp}) reduces to a scalar minimization problem. 
\begin{definition}
Let  $f:X\to \RR$. Problem (\ref{vp}) is said to be  Tykhonov  well-posed (T-wp for short)  if: 
\begin{itemize}
\item [1.] there exists a unique $\bar x \in X$ such that $f(\bar x) \le f(x)$, $\forall x \in X$; 
\item [2.] every sequence $x_n$ such that $f(x_n)\to \inf_X f$ is such that $x_n \to \bar x$. 
\end{itemize}
\end{definition}

Next proposition provides a useful characterization of Tykhonov well-posedness. It is called the Furi-Vignoli criterion \cite{fv1}. 

\begin{proposition}
Let $f: X \to \RR$ be lsc. The following alternatives  are equivalent: 

\begin{itemize}
\item [1.]Problem (\ref{vp}) is T-wp; 
\item [2.] ${\rm inf}_{ a>{\rm inf}_X f} \ {\rm diam}\ L_f(a)=0$, where  $L_f(a)=\{x\in X: f(x)\le a\}$. 
\end{itemize}
\end{proposition}


The following result regarding well-posedness of convex functions defined on a finite-dimensional space is well-known. 

\begin{theorem} (see e.g. \cite{dz})\label{th-wp-convex}
Let $X$ be finite-dimensional and $f: X \to \RR$ be  a convex function with a unique minimizer. Then problem (\ref{vp}) is T-wp. 
\end{theorem}

Theorem \ref{th-wp-convex} does not hold when $X$ is infinite-dimensional as the following example shows (see e.g. \cite{dz}). 

\begin{example}
Let $X$ be a separable Hilbert  space with orthonormal basis $\{e_n, n \in \NN\}$. Let $f(x)=\sum_{n=1}^{+\infty}\frac{\langle x, e_n\rangle^2}{n^2}$. Then $f$ is continuous, convex and has $\bar x=0$ as  unique minimizer, but problem (\ref{vp})  is not T-wp. Indeed the sequence $\sqrt n e_n$ is an unbounded minimizing sequence. 
\end{example}

Consider now the space 
$$
\Gamma : = \left\{f:X\to \RR : f \mbox{ is } \mbox{convex and lower semicontinuous} \right\}.
$$ 
We endow $\Gamma$ with a distance  compatible with the uniform convergenge on bounded sets. Fix $\theta \in X$ and set for any two functions $f,g \in \Gamma$ and $i \in \NN$,
$$
\norm{f-g}_{i} = \sup_{\norm{x - \theta} \leq i} |f(x) - g(x)|.
$$
If $\norm{f-g}_{i} = \infty$ for some $i$, then set $d(f,g) = 1$, otherwise
$$
d(f,g) = \sum_{i=1}^{\infty}2^{-i}\frac{\norm{f-g}_{i}}{1+\norm{f-g}_{i}}.
$$

When $X$ is infinite-dimensional, it can be shown that the set of functions $f\in \Gamma$ such that  problem (\ref{vp}) is T-wp is "big" in the sense that contains a dense $G_{\delta}$ set (see e.g. \cite{lucchetti2})

\begin{theorem}\label{density-scalar}\cite{lucchetti2}
Let $X$ be a Banach space and consider the set $\Gamma$, equipped with the  topology of uniform convergence. Then  the set of functions $f\in \Gamma$ such that  problem (\ref{vp}) is  T-wp  contains a dense $G_{\delta}$ set . 
\end{theorem}

If the convexity assumption is dropped weaker variants  of Theorem \ref{density-scalar} hold, in which density of the class of  functions $f \in \Gamma$ such that problem (\ref{vp}) is T-wp is proven.  Next  results (see e.g. \cite{lucchetti2}) will be useful in the following. 

\begin{proposition}
Let $f : X \to \RR $,  assume  $f$ has a minimum point $\bar x\in X$ and let  $g(x) =f(x) + a\| x- \bar x\|$  with $a>0$. Then problem $(X, g)$ is T-wp.
\end{proposition}

\begin{theorem}{\bf (Ekeland's Variational Principle)}\label{evp}
Let $f : X \to \RR$  be a lsc, lower bounded function. Let $\varepsilon > 0$, $r > 0$ and $\bar x \in  X$ be such
that $f( \bar x) < \inf_X f + r\varepsilon$. Then, there exists $\hat x \in  X$ enjoying the following
properties:
\begin{itemize}
\item [1.] $\|\hat x- \bar x\| < r$;
\item [2.] $f(\hat x) < f(\bar x) - \varepsilon \|\bar x-\hat x\|$;
\item [3.] Problem $(X, g)$ with $g(x)=f(x) + \varepsilon \|\hat x- x\|$ is T-wp.
\end{itemize}
\end{theorem}

\subsection{Well-posedness for vector optimization problems}
Several generalizations  of the notion of well-posedness to vector functions have been proposed. We refer to \cite{mmr} for a survey on  the topic and a study of the relations among different well-posednsess concepts. In that paper vector well-posedness notions have been divided in two classes: pointiwise and global notions. Notions in the first  class  define the well-posedness of a vector problem with respect to a fixed  efficient solution, while in the global notions the set of efficient solutions or weakly efficient solutions is considered as a whole.   \ \\
In this paper we focus on the notion  of well-posedness due to Dentcheva and Helbig \cite{dh} (DH-well-posedness) which is a pointwise notion according to \cite{mmr}. 

\begin{definition}
Let $f:X\to \RR$. Problem (\ref{vp}) is said to be DH-well-posed (DH-wp for short) at $\bar x \in {\rm Eff}\  (X, f)$   if 

$$
\inf_{\alpha > 0} \diam L_f^C(f(\bar{x}) + \alpha c)=0, \ \ \forall c \in C,
$$
where $ L_f^C(f(\bar{x}) + \alpha c) = \{x \in X: f(x) \in  f(\bar{x}) + \alpha c- C\}$. 
\end{definition}

In \cite{mmr} it  has been proven that DH- well-posedness is the strongest among the pointwise well-posedness notions, that is if problem (\ref{vp}) is DH-wp at $\bar x \in X$ then it is well-posed at $\bar x$ according to the other pointwise well-posedness notions known in the literature. The  next result gives a useful characterization of DH--well-posedness. 

\begin{theorem}\cite{gmmn, mmr}\label{wp-oriented-distance}
Problem (\ref{vp}) is DH-well-posed at $\bar x \in \Eff(X, f)$ if and only if problem (\ref{sp}) is T-wp. 
\end{theorem}
The following theorem  (see \cite{mmr}) gives a generalization of Theorem \ref{th-wp-convex}. 

\begin{theorem}\label{th-conv-vector-finite}
Let $X$ and $Y$ be finite-dimensional. Assume  $f: X \to Y$ is a $C$-convex function,  $\bar x \in \Eff(X, f)$ and $f^{-1}(f(\bar x))=\{\bar x\}$. Then problem (\ref{vp}) is DH-wp at $\bar x$. 
\end{theorem}

DH-well-posedness imposes some restrictions on the set ${\rm Eff}\, (X, f )$. Indeed, if problem (\ref{vp}) is DH-well-posed at $\bar x\in {\rm Eff}\, (X, f )$ then $\bar x \in {\rm StEff}\, (X, f )$. This property is typical of the vector case and shows that most of the vector well-posedness notions require implicitly stronger properties than the simple good behavior of minimizing sequence

\begin{theorem}\label{steff}\cite {mmr}
If $F: X \to Y$ is continuous and problem (\ref{vp}) is DH-wp at $\bar x\in  {\rm Eff}\, (X, f )$, then $\bar x \in {\rm StMin}\, (X, f )$.
\end{theorem}

\section{Density of DH-well-posed functions}\label{density}

The first result in this section shows that if the set of functions 
$${ \mathcal H}=\{f: X \to Y: \Eff(X, f) \not = \emptyset\}$$ 
is endowed with the topology of uniform convergence on bounded sets, then the set of functions $g \in {\mathcal H}$ enjoying DH-wp properties is dense in ${\mathcal H}$.
 
\begin{theorem}\label{density-nonempty}
Let $f\in { \mathcal H}$. Then, for every $\bar x \in \Eff(X, f)$, there exists a sequence of functions $f_n: X \to Y$ such that $f_n \to f$ in the uniform convergence on bounded sets, $ \bar x \in \Eff(X, f_n)$ for every $n$ and problem  $(X, f_n)$ is DH-wp at $\bar x$. Further, if $f$ is continuous then  $\bar x \in {\rm StEff}\, (X, f_n)$ for every $n$.
\end{theorem}

\begin{proof}
Let $k^0 \in {\rm int}\, C$ be fixed and consider the sequence of functions 
\begin{equation*}
f_n(x)=f(x)+ \frac1n \|x- \bar x\|k^0
\end{equation*}
Since $\bar x \in \Eff(X, f)$, it holds
\begin{equation}\label{eff}
f(x)-f(\bar x)\not \in -C, \ \forall x \in X, \ x \not = \bar x
\end{equation}
Hence
\begin{equation*}
f(x)-f_n(\bar x)=f(x)-f(\bar x)+ \frac1n \|x-\bar x\|k^0\not \in -C, \ \forall x \in X, \ x \not =\bar x
\end{equation*}
since (\ref{eff}) holds. Hence  $\bar x \in \Eff(X, f_n) \ \forall n$. Since $\Eff(X, f_n)\subseteq \WEff(X, f_n)$, Theorem \ref{min-car} implies $D_{-C}(f_n(x)-f_n(\bar x)) \ge 0$ for every $ x \in X$.  Now we prove problem  $(X, f_n)$ is DH-wp at $\bar x$ for every $n$. From Theorem \ref{wp-oriented-distance} we know that problem $(X, f)$ is DH-wp at $\bar x \in X$ if and only if the scalar problem (\ref{sp}) is T-wp at $\bar x$. Since ${\rm int}\, C \not = \emptyset$, $C^*$ has a closed convex ${\rm weak}^*$-compact base  
\begin{equation}
G=\{\xi \in C^* : \langle \xi , k^0\rangle=1 \}
\end{equation}
(see e.g. \cite{jameson}). 
According to \cite{mmr} there exists a constant $\alpha >0$ such that 
\begin{equation*}
D_{-C}(f_n(x)-f(\bar x))\ge \alpha  \max_{\xi \in G}\langle  \xi, f_n(x)-f_n(\bar x)\rangle 
\end{equation*}
\begin{equation*}
=\alpha \max _{\xi \in G}\langle \xi , f(x)-f(\bar x)+ \frac1n \|x -\bar x\|k^0 \rangle 
\end{equation*}
\begin{equation*}
=\alpha \max_{\xi \in G}\langle \xi , f(x)-f(\bar x)\rangle+ \frac1n \|x-\bar x\|
\end{equation*}
For a fixed $n$, let $x_k$ be a minimizing sequence for $D_{-C}(f_n(x)-f_n(\bar x))$, that is $D_{-C}(f_n(x_k)-f_n(\bar x))\to 0$. If $x_k \not \to \bar x$ we get 
\begin{equation*}
D_{-C}(f_n(x_k) - f_n(\bar x)) \ge \alpha \max_{\xi \in G}\langle \xi , f(x_k)-f(\bar x)\rangle+ \frac1n \|x_k-\bar x\|
\end{equation*}
\begin{equation*}
\ge \inf_{k \in \NN} \frac 1n \|x_k-\bar x\|>0
\end{equation*}
which contradicts to $x_k$ minimizing sequence for $D_{-C}(f_n(x)-f_n(\bar x))$ (the last inequality follows since $\bar x \in \Eff(X, f)$ implies $\max_{\xi \in G}\langle \xi , f(x_k)-f(\bar x)\rangle \ge 0 \ \ \forall x \in X$) . Hence $x_k \to \bar x$ and problem  $(X, f_n)$ is DH-wp at $\bar x$. Finally, we get the desired result  observing that $f_n \to f$ in the uniform convergence on bounded sets. If $f$ is continuous, then apply Theorem \ref{steff} to conclude the proof. 
\end{proof}

To prove the second density result in this section  we need the following definition and the next lemma.  

\begin{definition} \label{bounded}\cite{gmmn}
We say that $f: X \to Y$ is $C$-bounded from below by $\xi \in C^* \backslash \{0\}$ when $\inf_{x \in X} \langle \xi , f(x)\rangle> - \infty$. 
\end{definition}

Let $\bar x \in \Eff(X, f)$, consider function
$$
h_{\bar \xi}(x) = \langle \bar \xi , f(x)\rangle 
$$ 
and the associated scalar minimization problem 
\begin{equation*}\label{lsp}\tag{$X, h_{\bar \xi}$}
\min h_{\bar \xi}(x) \ , \ x \in X
\end{equation*}

\begin{lemma}\label{wp-scalar}
Let $\bar x \in \Eff(X, f)$ and  $\bar \xi \in C^*\backslash \{0\}$. If problem (\ref{lsp})  is T-wp at $\bar x$, then problem (\ref{vp}) is DH-wp at $\bar x$. 
\end{lemma}
\begin{proof}
Without loss of generality let  $\bar \xi \in C^*\cap \partial B$.  Assume problem  (\ref{vp})  is not DH-wp at $\bar x$. Since $\bar x \in \Eff(X, f)\subseteq {\rm WEff}(X, f)$, by Theorem \ref{min-car} it holds $D_{-C}(f(x)-f(\bar x))\ge 0$, for every $x \in X$ and by Theorem \ref{wp-oriented-distance} problem (\ref{sp}) is not T-wp. Then there exists a sequence $x_n \in X$ such that $D_{-C}(f(x_n)-f(\bar x))\to 0$ but $x_n \not \to \bar x$. Since, by Proposition \ref{properties}
\begin{equation*}
D_{-C}(f(x_n)-f(\bar x))=\max_{\xi \in C^* \cap \partial B}\langle \xi , f(x_n)-f(\bar x)\rangle 
\end{equation*}
it follows 
\begin{equation*}
 D_{-C}(f(x_n)-f(\bar x)) \ge \langle \bar \xi , f(x_n)-f( \bar x)\rangle = h_{\bar \xi}(x_n) -h_{\bar \xi}(\bar x)
\end{equation*}
Since problem (\ref{lsp}) is T-wp at $\bar x$, it follows that $\bar x$ is a minimum point for $h_{\bar \xi}$ and hence  $ h_{\bar \xi}(x_n) -h_{\bar \xi}(\bar x)=  \langle \bar \xi , f(x)-f(\bar x)\rangle\ge 0\ \forall n$.  From $D_{-C}(f(x_n)-f(\bar x))\to 0$ it follows  $h_{\bar \xi}(x_n)\to h_{\bar \xi}(\bar x)$ which contradicts problem (\ref{lsp}) is T-wp  since $x_n \not \to \bar x$. 
\end{proof}

In the  next result we drop the asumption ${\rm Eff}\, (x, f) \not= \emptyset$ and we show that if the  set of functions 
\begin{equation*}
{\mathcal H'}=\{f: X \to Y: \exists \ \xi \in C^*   \  {\rm such\  that}\ f \ {\rm is}\  C-{\rm bounded \ from\  below\  by}\  \xi \}
\end{equation*}
 is endowed with the topology of uniform convergence on bounded sets, then the set of functions $g \in {\mathcal H'}$ enjoying DH-wp properties is dense in ${\mathcal H'}$. 
We endow ${\mathcal H'}$ with a distance  compatible with the uniform convergenge on bounded sets (see e.g. \cite{lucchetti2}). Fix $\theta \in X$ and set for any two functions $f,g \in {\mathcal H'}$ and $i \in \NN$,
\begin{equation*}
\norm{f-g}_{i} = \sup_{\norm{x - \theta} \leq i} \norm{f(x) - g(x)}.
\end{equation*}
If $\norm{f-g}_{i} = \infty$ for some $i$, then set $d(f,g) = 1$, otherwise
\begin{equation}\label{distance}
d(f,g) = \sum_{i=1}^{\infty}2^{-i}\frac{\norm{f-g}_{i}}{1+\norm{f-g}_{i}}.
\end{equation}

\begin{theorem}\label{density-C-bounded}
Assume there exists $\bar \xi \in C^*\backslash \{0\}$ such that $f:X\to Y$ is $C$-bounded from below by $\bar \xi$ and $\langle \bar \xi , f(x)\rangle$ is lsc with respect to $x \in X$. Then, there exists a sequence of functions $f_n : X \to Y$ uniformly converging to $f$ on the bounded sets, such that $\Eff(X, f_n) \not = \emptyset$ for every $n$ and problem $(X, f_n)$ is DH-wp at some $\bar  x_n\in \Eff(X, f_n)$.  
\end{theorem}
\begin{proof}
Fix $\sigma >0$ and take $j$ so large that setting $g(x)=f(x)+\frac1j\|x-\theta\|k^0$ with $k^0 \in {\rm int}\, C$ such that $\langle \bar \xi , k^0\rangle=1$, it holds $d(f, g)<\frac{\sigma}{2}$. Now set  
\begin{equation}
g_{\bar \xi}(x)  = \langle \bar \xi , g(x)\rangle =  \langle \bar \xi , f(x)\rangle+ \frac1j \|x-\theta\| 
\end{equation}
and observe that $g_{\bar \xi}(x)$ is lower bounded  by Definition \ref{bounded}.  Hence, we can find $M>0$ such that 
\begin{equation*}
\{x \in X: g_{\bar \xi}(x) \le \inf_{x \in X} g_{\bar \xi}(x) +1\} \subseteq B(\theta , M)
\end{equation*}
Let $s= \sum_{k=0}^{+\infty}\frac{1}{2^k} (k+M)\|k^0\|$ and apply Theorem \ref{evp} with $\varepsilon = \frac{\sigma}{2s}$ and arbitrary $r$ to find a point $\bar x=\bar x_{\sigma}\in X$ such that $\|\bar x - \theta\|\le M$,  $\bar x$ is the unique minimizer of 
\begin{equation*}
h_{\bar \xi}(x)=\langle \bar \xi , g(x)\rangle + \varepsilon \|x-\bar x\|
\end{equation*} 
and problem (\ref{lsp}) is T-wp at $\bar x$.  Let  
\begin{equation*}
h(x)= g(x)+ \varepsilon \|x- \bar x\|k^0
\end{equation*}
and observe that since $\bar x$ minimizes $h_{\bar \xi}(x)$, it holds  
\begin{equation*}
h_{\bar \xi}(x)-h_{\bar \xi}(\bar x)= \langle \bar \xi , h(x)-h(\bar x) \rangle>0, \ \forall x\in X\backslash\{\bar x\}
\end{equation*}
which  implies
\begin{equation*}
D_{-C}(h(x)-h(\bar x)) = \max_{\xi \in C^* \cap \partial B}\langle  \xi ,  h(x)-h(\bar x)\rangle \ge 
\end{equation*}
\begin{equation*}
\langle \bar \xi , h(x)-h(\bar x)\rangle >0, \ \forall x\in X\backslash\{\bar x\}
\end{equation*}
Hence, Theorem \ref{min-car} implies $\bar x \in \Eff(X, h)$. Combining  Theorem \ref{min-car}  and    Lemma \ref{wp-scalar}, we obtain that  problem 
\begin{equation*}\label{hp}\tag{$X, h$}
\min h(x) \ , \ x\in X
\end{equation*}
is DH-wp at $\bar x$. Now observe that 
\begin{equation*}
\|h(x)-g(x)\|_i \le  \varepsilon \|k^0\|(i+M)
\end{equation*}
It follows $d(h, g) \le \varepsilon s= \frac{\sigma}{2}$ and then  $d(f, h) < \sigma$. Take now $\sigma = \frac1n, n=1,2, \ldots$ and set $\bar x_n =\bar x_{\sigma}$ to complete the proof. 
\end{proof}

The next result shows that under some hypotheses, the assumption in Theorem \ref{density-C-bounded} is weaker than the assumption in Theorem \ref{density-nonempty}. We recall the following fundamental result. 

\begin{theorem}\label{sion}({\bf Sion's Minimax Theorem \cite{sion, Tuy}})
Let $Z$ be a compact convex subset of a linear topological space and $W$ a convex subset of a linear topological space. Let $g$ be a real-valued function on $Z\times W$ such that 
\begin{itemize}
\item [i)] $g(\cdot, w)$ is upper semicontinuous and quasi-concave on $Z$ $\forall w \in W$; 
\item [ii)] $g( z , \cdot)$ is lower semicontinuous and quasi-convex on $W$ $\forall z \in  Z$.
\end{itemize}
Then
\begin{equation*}
{\rm sup}_{z\in Z}{\rm inf}_{w \in W}f(z, w)= {\rm inf}_{w \in W}{\rm sup}_{z\in Z}f(z, w)
\end{equation*}
\end{theorem}

\begin{proposition}
Let $f:X \to Y$ be $*$-quasiconvex  and $C$-lsc with respect to $x \in X$,  for every $\xi \in C^*$. Then, if $\Eff(X, f) \not = \emptyset$,  there exists $ \bar \xi \in C^*\backslash\{0\}$ such that $f$ is $C$-bounded from below by $\bar \xi$. 
\end{proposition}
\begin{proof}
Assume $\Eff(X, f) \not = \emptyset$ and let  $\bar x \in \Eff(X, f)$. Ab absurdo assume that  for every $\xi \in C^*\backslash\{0\}$ it holds 
\begin{equation*}
\inf_{x \in X}\langle \xi , f(x)\rangle=\inf_{x \in X}\langle \xi , f(x)- f(\bar x)\rangle= -\infty
\end{equation*}
Since $ {\rm int}\, C \not = \emptyset$, $C^*$ has a weak$^*$-compact base $G$.  Function $g(\xi, x)=\langle \xi , f(x)- f(\bar x)\rangle$, $\xi \in G$, $x\in X$, is linear and continuous with respect to $\xi$ and quasiconvex  with respect to $x$. Further, since $f$ is $C$-lsc with respect to $x \in X$, $g(\xi, x)$ is lsc with respect to $x\in X$. Since $\bar x \in \Eff(X, f)$, it holds $\max_{\xi \in G} \langle \xi , f(x)-f(\bar x)\rangle\ge 0$ for every $x \in X$. Apply Sion's Minimax Theorem  to get the following chain of equalities
\begin{equation*}
-\infty= \sup_{\xi \in G}\inf_{x \in X}\langle \xi , f(x)-f(\bar x)\rangle=\inf_{x \in X} \sup_{\xi \in G}\langle \xi , f(x)-f(\bar x)\rangle
\end{equation*}
which implies there exists $\tilde x\in X$ such that $\sup_{\xi \in G}\langle \xi , f(\tilde x)-f(\bar x)\rangle<0$. A contradiction to $\bar x \in \Eff(X, f)$. 
\end{proof}

Generalized convexity assumptions in the previous reult cannot be removed as the following example shows. 

\begin{example}
Let $X= \RR$, $Y=\RR^2$, $C=C^*= \RR^2_+$, $f:X \to Y$ defined as $f(x)= (x, -xe^x)$ is not $*$-quasiconvex.  We have ${\rm Eff} (X, f)= [0, +\infty)\not = \emptyset$ but for any $\xi \in C^*\backslash \{0\}$ we have  ${\rm inf}_{x \in X }\langle \xi , f(x) \rangle= -\infty$.  Hence does not exist $\xi \in C^*\backslash \{0\}$ such that    $\langle \xi , f(x) \rangle$ is bounded from below. 
\end{example}
\section{Genericity of DH-well-posedness for $C$-convex functions}\label{Genericity}

In this section we show that the set of  $C$-convex and $C$-lsc functions enjoying DH-well-posedness properties contains a dense $G_{\delta}$ set. To prove the main  theorem in this section we need some preliminary results. 

\begin{proposition}\label{co+qu}
Let $f: X\to{\RR}$ a  convex and lsc function,   $\bar x \in X$ and set   $g(x) = f(x) + a\|x-\bar x\|^{\alpha}, \ a>0, \alpha \ge 1$. Then $\lim_{\|x\|\to +\infty} g(x)=+\infty$. Furthermore $g(x)$ is lower bounded. 
\end{proposition}

\begin{proof}
We prove that for every sequence $x_n\in X$ with $\|x_n\|\to +\infty$ it holds $\lim_{n\to +\infty}g(x_n)=+\infty$. Denote by $X^*$ the topological dual space of $X$. 
Since $f(x)$ is convex, the set $\partial f\left(\bar x\right)\subseteq X^*$ of all subgradients of $f$ at $\bar x$ is nonempty and by definition of subgradient \cite{grtz}, for every continuous linear functional $v \in \partial f\left(\bar x\right)$ it holds $f(x) \geq f\left(\bar x\right) + v(x-\bar x), \forall x \in X$. Hence, 
\begin{equation*}\begin{split}
\lim_{n\to +\infty} g\left(x_{n}\right)  &= \lim_{n\to +\infty} \left[f\left(x_{n}\right) +a \norm{x_{n}-\bar x}^{\alpha} \right]\\
&\geq \lim_{n\to +\infty} \left(f\left(\bar x\right) + v(x_{n}-\bar x) +a \norm{x_{n}-\bar x}^{\alpha} \right)\\
&=\lim_{n\to +\infty}\left[f(\bar x)+\norm{x_{n}-\bar x}^{\alpha}\left(v\left(\frac{x_n-\bar x}{\|x_n-\bar x\|}\right)\|x_n-\bar x\|^{1-\alpha}+a\right)\right]\\
&= +\infty\ . 
\end{split}\end{equation*}
(the last equality follows since a continuous linear functional is bounded). To prove that $g(x)$ is lower bounded observe that for every $M\in \RR$, there exists $k>0$ such that $g(x)>M$ for $\|x\|>k$. If we take $A=\{x \in X: \|x\|\le k\}$, $g(x)$ is lower bounded on the bounded set $A$ (see e.g. \cite{lucchetti2}), which concludes the proof. 
\end{proof}

\begin{corollary}\label{convex-infinity}
Let $f:X \to Y$ be a $C$-convex, $C$-lsc function and for  $\xi \in C^*\backslash\{0\}$ and $a>0, \alpha \ge 1$ set $g_{\xi}(x)= \langle \xi , f(x)\rangle+ a\|x-x^0\|^{\alpha}$.  Then,  $\lim_{\|x\|\to +\infty} g_{\xi}(x)=+\infty$ and $g$ is lower bounded.
\end{corollary}
\begin{proof}
The proof follows from Proposition \ref{co+qu} since $f$ $C$-convex and $C$-lsc implies $g$ convex and lsc for every $\xi \in C^*\backslash\{0\}$.  
\end{proof}
Let ${\mathcal F}$ be the set of $C$-convex and $C$-lsc functions $f:X\to Y$. We endow ${\mathcal F}$ with the distance defined by (\ref{distance}), compatible with the topology of uniform convergence on bounded sets. 

\begin{theorem}
Let ${\mathcal F}$ be the set of $C$-convex and $C$-lsc functions endowed with the topology of uniform convergence on bounded sets and let $\tilde {\mathcal F}$ be the set of functions   $f \in {\mathcal F}$  such that $ \Eff(X, f)\not =\emptyset$ and  problem  (\ref{vp}) is  DH-wp  at some point $\bar x \in \Eff(X, f)$. Then $\tilde {\mathcal F}$ contains a dense $G_{\delta}$ set .  
\end{theorem}

\begin{proof} 
The initial argument  of the proof is inspired  to that of Theorem 2.1 in \cite{luc-mig}. 
If we fix $k^0 \in {\rm int}\, C$, we can find  $\bar \xi \in C^{*}$ such that $\langle \bar \xi , k^0\rangle =1$. Consider the set
\begin{equation*}
{\mathcal Z}= \{z: X \to \RR\  {\rm such \  that}\  z(x)= \langle \bar \xi  , f(x) \rangle , f \in {\mathcal F} \}
\end{equation*}
Since $f$ is $C$-lsc,  $z$ is lsc. Endow ${\mathcal Z}$ with the topology of uniform convergence on bounded sets and let $S: {\mathcal F} \to {\mathcal Z}$ be the map $S(f)= z$, with $z$ defined as before. Then $S$ is a continuous map. Let 
\begin{equation}
{\mathcal A}_n=\{z\in {\mathcal Z} : \exists a > \inf_{x \in X} z , \ {\rm diam}\, L_z(a)< \frac1n\}
\end{equation}
where $L_z(a) =\{x \in X : z(x) \le a\}$. Observe that $L_z(a)$ are closed convex sets since $z$ is convex and lsc. Then, it is known (see e.g. \cite{lucchetti2}) that if $z_n \to z$ in the uniform convergence, then ${\rm diam}\, L_{z_n}(a) \to {\rm diam}\, L_z(a)$, which gives continuity of the ${\rm diam}$ function. Hence ${\mathcal A}_n$ is an open set for all $n$ and this implies $S^{-1}({\mathcal A}_n)$ is an open set for all $n$.  
We claim that the set ${\mathcal W}$ of those functions $h \in {\mathcal F}$ such that problem $(X, S(h))$ is T-wp is dense in ${\mathcal F}$. Since 
\begin{equation*}
{\mathcal W}= \bigcap_{n=1}^{+\infty}S^{-1}({\mathcal A}_n)
\end{equation*}
then it is a $G_{\delta}$ set i.e. the countable intersection of open   sets. Let $f \in {\mathcal F}$, $\sigma >0$ and take $j$ so large that setting $g(x)=f(x)+\frac1j\|x-\theta\|k^0$  it holds $d(f, g)<\frac{\sigma}{2}$ . Setting  $g_{\bar \xi}(x)= \langle \bar \xi , g(x)\rangle$   we have $\lim_{\|x\|\to +\infty}g_{\bar \xi}(x) = +\infty$ and $g_{\bar \xi}$ is lower bounded by Corollary \ref{convex-infinity}. The proof now follows along the lines of Theorem \ref{density-C-bounded}. We can find $M>0$ such that 
\begin{equation*}
\{x \in X: g(x) \le \inf_{x \in X} g(x) +1\} \subseteq B(\theta , M)
\end{equation*}
Let  $h: X \to Y$ be defined as
\begin{equation*}
h(x)= g(x)+ \varepsilon \|x- \bar x\|k^0
\end{equation*}
and let $s= \sum_{k=0}^{+\infty}\frac{1}{2^k} (k+M)\|k^0\|$. Apply Theorem \ref{evp} with $\varepsilon = \frac{\sigma}{2s}$ and arbitrary $r$ to find a point $\bar x=\bar x_{\sigma}\in X$ such that $\|\bar x - \theta\|\le M$, $\bar x$ is the unique minimizer of 
\begin{equation*}
S(h)(x)=\langle \bar \xi , g(x)\rangle + \varepsilon \|x-\bar x\|
\end{equation*} 
and problem $(X, S(h))$ is T-wp at $\bar x$  and hence $h \in {\mathcal W}$.
This implies  that problem $(X, h)$  is DH-wp at $\bar x$ by Lemma \ref{wp-scalar}. Now observe that 
\begin{equation*}
\|h(x)-g(x)\|_i \le  \varepsilon \|k^0\|(i+M)
\end{equation*}
It follows $d(h, g) \le \varepsilon s= \frac{\sigma}{2}$ and then $d(f, h) < \sigma$. 
 Hence ${\mathcal F}$ contains a dense $G_{\delta}$ set, which concludes the proof. 
\end{proof}



\end{document}